\documentclass[runningheads]{llncs}
\usepackage[T1]{fontenc}
\usepackage{graphicx}
\usepackage[title]{appendix}
\usepackage{amsmath}
\usepackage{amsfonts,amssymb}
\usepackage[electronic]{ifsym}
\usepackage{authblk}
\usepackage{hyperref}
\usepackage[nameinlink, capitalize]{cleveref}
\usepackage[shortlabels]{enumitem}
\usepackage{esvect}

\usepackage{tikz}
\usetikzlibrary{automata,positioning,arrows.meta}

\newcommand{\N}{\mathbb{N}}

\newcommand{\BA}{\mathsf{BA}}
\newcommand{\pr}{\mathrm{pr}}
\newcommand{\forallP}{\forall^{P_k}}
\newcommand{\existsP}{\exists^{P_k}}

\newcommand{\T}{\mathsf{T}}
\newcommand{\one}{\underline{1}}

\begin{document}
\title{A natural axiomatization of B{\"u}chi Arithmetic}
%
%
\author{Konstantin Kovalyov\inst{1}\orcidID{0009-0004-6621-9184}}

\authorrunning{K. Kovalyov}

\institute{Steklov Mathematical Institute of Russian Academy of Sciences, Moscow, Russia\\
\email{kovalyov.ka@mi-ras.ru}}

\maketitle

\begin{abstract}
    We investigate Büchi Arithmetic $\BA_k$ --- the elementary theory of the natural numbers equipped with addition and the function mapping a number $x$ to the greatest power of $k$ dividing $x$. $\BA_k$ is known to be decidable and to enjoy a few important properties, in particular, a first-order structure is automatic iff it is interpretable in $\BA_k$. We propose a natural axiomatization of this theory based on a comprehension schema restricted to bounded formulas, interpreting natural numbers as finite (multi)sets of powers of $k$ via their base-$k$ expansions. The completeness proof for this axiomatization proceeds through a formalization of the Büchi–Bruyère Theorem on the equivalence of definability in Büchi Arithmetic and recognizability by finite automata.
\end{abstract}

\section{Introduction}\label{section_introduction}

Base-$k$ Büchi arithmetic, denoted $\mathsf{BA}_k$, is the elementary theory of the structure $(\mathbb{N}, 0, S, +, V_k, \leq)$, where $k \geq 2$ is a fixed integer. For $x > 0$, $V_k(x)$ denotes the largest power of $k$ dividing $x$, and $V_k(0) = 0$.  It is useful to think about $V_k(x)$ as the position of the least significant non-zero digit of the base-$k$ expansion of $x$. The structure $(\mathbb{N}, 0, S, +, V_k, \leq)$ is \emph{automatic} and, moreover, it interprets every automatic structure (see \cite[Theorem~4.5]{Blumensath2004}). The theory $\mathsf{BA}_k$ is well known for the following fundamental result.

\begin{theorem}[{\cite{Buchi,bruyare}}]\label[theorem]{theorem_Buchi_Bruyere}
    A set $A \subseteq \mathbb{N}^m$ is definable in $\mathsf{BA}_k$ if and only if the set of base-$k$ expansions of the elements of $A$ is recognizable by a finite automaton. Moreover, the construction of a finite automaton from a defining formula is effective; consequently, $\mathsf{BA}_k$ is decidable.
\end{theorem}

Büchi Arithmetic has been the subject of extensive investigation, particularly in connection with problems of definability, complexity, and interpretability. A central theme in this area is the relationship between Büchi Arithmetic and Presburger Arithmetic, which also highlights the close link with recognizability by finite automata. This connection is made precise by a classical result due to A.~Cobham and A.~Semenov.

\begin{theorem}[{\cite{Cobham_1969,Semenov_1977}}]
    For $A \subseteq \N^m$, the following are equivalent:
    \begin{enumerate}[start=1,label={\rm (\roman*)}]
        \item $A$ is definable in Presburger Arithmetic;
        \item $A$ is definable in $\BA_k$ for all $k \geq 2$;
        \item $A$ is definable in $\BA_k$ and $\BA_l$ for some multiplicatively independent $k, l \geq 2$.
    \end{enumerate}
\end{theorem}

This result has been substantially extended in subsequent work, most notably in \cite{Bes_1997} and, more recently, in \cite{Hieronymi_Schulz_2022}. Namely, for multiplicatively independent $k, l \geq 2$, if $X$ and $Y$ are not definable in Presburger Arithmetic, $X$ is definable in $\BA_k$, and $Y$ is definable in $\BA_l$, then the theory of the expansion of Presburger Arithmetic with the predicates for $X$ and $Y$ is undecidable.

From the point of view of computational complexity, $\BA_k$ has a complicated behavior as the following results describe. 

\begin{theorem}[follows from \cite{Meyer_1975}]
    $\BA_k$ is not Kalmar elementary. That is, $\BA_k$ cannot be decided in $$\exp_2(\exp_2( \dots \exp_2(n) \dots ))$$ steps for any fixed number of nested exponentiations (here $\exp_2(n) = 2^n$). 
\end{theorem}

Note that this bound is, in a sense, sharp: for a formula of length $n$ with $m$ quantifier alternations, the method from \cref{theorem_Buchi_Bruyere} runs in $$\underbrace{\exp_2(\exp_2( \dots \exp_2(}_{m \text{ times}} \mathrm{poly}(n)) \dots ))$$ steps. Despite this extreme worst-case complexity, certain fragments of $\BA_k$ admit more feasible decision procedures.

\begin{theorem}[\cite{Guepin_Haase_Worell_2019}]
    The existential fragment of $\BA_k$ is $\mathsf{NP}$-complete.
\end{theorem}

Nevertheless, some purely logical questions about Büchi Arithmetic remain open. In particular, no natural axiomatization of this theory is known. In our earlier work \cite{kovalyov2024,kovalyov2025_trudy}, we suggested a simple axiomatization of $\BA_k$ based on bounds for witnesses for existential quantifiers. However, the axioms are not very natural, and it would be difficult to verify whether an arbitrary structure is a model of these axioms. The following related negative result was obtained by A.~Zapryagaev \cite{zapryagaev}.

\begin{theorem}
    The theory axiomatized as follows is not complete:
        \begin{enumerate}[start=1,label={\rm (\roman*)}]
            \item axioms of Presburger Arithmetic;
            \item $V_2(x) = 0 \leftrightarrow x = 0$;
            \item $\neg \exists y \ (x = y + y) \rightarrow V_2(x) = 1$;
            \item $x = t + t \rightarrow V_2(x) = V_2(t) + V_2(t)$.
        \end{enumerate}
\end{theorem}

The absence of an adequate axiomatization of $\BA_k$ makes the study of model-theoretic and interpretability questions difficult. For instance, no explicit construction of a nonstandard model of $\BA_k$ is currently known.

In this paper, we provide a natural axiomatization of $\BA_k$ based on basic properties of $+$ and $V_k$, together with a comprehension scheme restricted to $\Delta_0$-formulas (see below for the definition of $\Delta_0$-formulas). Here, natural numbers are interpreted as finite (multi)sets of powers of $k$ via their base-$k$ expansions (see below). Alternatively, the comprehension scheme can be replaced by an induction scheme for powers of $k$ (see \cref{remark_on_induction}). The main idea is to express, by means of a $\Delta_0$-formula, the existence of a run of a finite automaton (cf. \cite{VILLEMAIRE,haase}) and to formalize within this axiomatization a proof of \cref{theorem_Buchi_Bruyere}.

Let us now describe our axiomatization in detail. Fix a natural number $k \geq 2$ as the base for Büchi Arithmetic. For each $n \in \N$, we denote by $\underline{n}$ the term
\[
\underbrace{S(S(\dots S}_{n \text{ times}}(0)\dots))
\]
called the \emph{numeral} corresponding to $n$. For a term $t$ and $n \in \N$, we denote by $n t$ the term
\[
\underbrace{t + (t + (\dots (t}_{n \text{ times}} + 0)\dots)).
\]
A formula $\varphi$ is said to be \emph{bounded}, or a \emph{$\Delta_0$-formula}, if every quantifier occurring in $\varphi$ is of the form $\exists x \leq t$ or $\forall x \leq t$, where $x$ does not occur in $t$. We denote by $\Delta_0$ the set of bounded $\mathcal L_{\BA}$-formulas. We denote by $\Pi_1$ the set of formulas of the form $\forall \vv x \ \theta$, where $\theta \in \Delta_0$. Finally, we write $\varphi(x_1, \dots, x_n)$ to indicate that the free variables of $\varphi$ are among $\{x_1, \dots, x_n\}$.

We introduce the following formulas:
\begin{align*}
        P_k(d) & \ : \iff (d > 0) \wedge (V_k(d) = d), \\
        x \equiv_k y & \ : \iff \exists z \leq x + y \ (x = y + kz \vee y = x + kz),
\end{align*}
and, for each $a \in \{0, \dots, k - 1\}$, a formula
\begin{align*}
        (x)_d = a & \ :\iff \ P_k(d) \wedge \exists y < d \ \exists z \leq x\ \Big((x = y + ad + z) \wedge (z = 0 \vee V_k(z) > d) \Big).
\end{align*}
    
Here $P_k$ defines the set of powers of $k$, and $(x)_d = a$ expresses that the coefficient before $d$ in the base-$k$ expansion of $x$ is equal to $a$. We write $\exists^{P_k} d \ \varphi$ and $\forall^{P_k} d \ \varphi$ as abbreviations for $\exists d \,(P_k(d) \wedge \varphi)$ and $\forall d \,(P_k(d) \rightarrow \varphi)$, respectively. Also, we will sometimes write $d \in P_k$ instead of $P_k(d)$ when convenient.

\begin{definition}
    A theory $\T_k$ consists of the universal closures of the following formulas:
    \begin{itemize}
        \item[{\rm (Add)}] axioms of discretely ordered commutative cancelative monoids;

        \item[{\rm (Ord0)}] $x \geq 0$;

        \item[{\rm (Ord1)}] $x \leq y \leftrightarrow \exists z \leq y\  (z + x = y)$;

        \item[{\rm (Mod)}] $\bigvee_{a = 0}^{k - 1} (x \equiv_k \underline{a})$;

        \item[{\rm (V0)}] $V_k(0) = 0$, $V_k(\underline 1) = \underline 1$;

        \item[{\rm (V1)}] $V_k(k x) = k V_k(x)$, $P_k(d) \to \bigwedge _{a = 1}^{k - 1} (V_k(a d) = d)$;

        \item[{\rm (V2)}] $(x > 0 \wedge y > 0 \wedge V_k(x) < V_k(y)) \to V_k(x + y) = V_k(x)$;

        \item[{\rm (V3)}] $x > 0 \to \exists y \leq x \ \bigvee _{a = 1}^{k - 1} \big(x = y + a V_k(x) \wedge (V_k(y) > V_k(x) \vee y = 0)\big)$;

        \item[{\rm (V4)}] $ x > 0 \to  \existsP d \leq x \ ( d \leq x < k d)$;

        \item[{\rm (V5)}] $V_k(V_k(x)) = V_k(x);$

        \item[{\rm (Comp)}] $P_k(d) \to \exists x < d \ \forallP d' < d \ \Big( \big((x)_{d'} = 1 \land \varphi(d', \vv y) \big) \vee \big((x)_{d'} = 0 \land \neg \varphi(d', \vv y) \big)\Big)$ for every $\Delta_0$-formula $\varphi(d, \vv y)$.
    \end{itemize}
\end{definition}

Let us clarify the meaning of some of the axioms. Axiom (Ord1) states that subtraction is well-defined. Axiom (V3) asserts that, for $x > 0$, there exists a coefficient before $V_k(x)$ in the $k$-expansion of $x$. Axiom scheme (Comp) corresponds to \emph{comprehension} when we interpret numbers as finite sets of powers of $k$. In \cref{section_axiomatization}, we prove the completeness of $\T_k$ (see \cref{main_theorem}).

\section{Preliminaries}\label{section_preliminaries}
We fix a natural number $k \geq 2$ as the base of Büchi arithmetic. By $\mathcal L_{\BA}$ we denote the language of Büchi arithmetic, namely $(0, S, +, V_k, \leq)$. We denote the standard model $(\N, 0, S, +, V_k, \leq)$ simply by $\N$. The theory $\BA_k$ is the elementary theory of $(\N, 0, S, +, V_k, \leq)$.

By a \emph{finite automaton} (FA for short) over a finite alphabet $\Sigma$ we mean a tuple $A = (\Sigma, Q, q_0, F, \Delta)$, where
\begin{enumerate}[(i)]
    \item $Q$ is a finite set of states;
    \item $q_0 \in Q$ is the initial state;
    \item $F \subseteq Q$ is the set of accepting states;
    \item $\Delta \subseteq Q \times \Sigma \times Q$ is the transition relation.
\end{enumerate}
Given a finite automaton $A$, we denote by $L(A)$ the language recognized by $A$, namely 
$$L(A) \ := \ \{w \in \Sigma^* \mid \text{some accepting state is reachable from } q_0 \text{ when } A \text{ reads }w\}.$$

A finite automaton $A = (\Sigma, Q, q_0, F, \Delta)$ is called a \emph{deterministic finite automaton} (DFA for short) if $\Delta$ is a total function $Q \times \Sigma \to Q$. Given a DFA $A$, a state $q \in Q$, and a word $w \in \Sigma^*$, there exists a unique state $q' \in Q$ reached from $q$ after $A$ reads $w$; we denote this state by $A(q, w)$. 

A language $L \subseteq \Sigma^*$ is said to be \emph{recognizable} if there exists a FA (or, equivalently, a DFA) $A$ such that $L = L(A)$.

We denote by $\Sigma_k$ the alphabet $\{0, \dots, k - 1\}$. We say that a word $w \in \Sigma_k^*$ is a \emph{$k$-expansion} of $a \in \N$ if
\[
\sum_{i < |w|} w_i k^i = a.
\]
Similarly, we say that a word $w \in (\Sigma_k^m)^*$ is a \emph{$k$-expansion} of $(a_1, \dots, a_m) \in \N^m$ if
\[
\sum_{i < |w|} (w_i)_j k^i = a_j
\]
for every $j = 1, \dots, m$. Here $w_i$ denotes the $i$-th symbol of $w$ (that is, an element of $\Sigma_k^m$), and $(w_i)_j$ denotes the $j$-th component of $w_i$ (that is, an element of $\Sigma_k$).

Given a set $A \subseteq \N^m$, we denote by $[A]_k$ the language consisting of all $k$-expansions of elements of $A$. We say that $A \subseteq \N^m$ is \emph{$k$-recognizable} if the language $[A]_k$ is recognizable (as a language over $\Sigma_k^m$).

Given finite automata $A = (\Sigma_k^m, Q, q_0, F, \Delta)$ and
$A' = (\Sigma_k^m, Q', q_0', F', \Delta')$, we use the following notation:
\begin{itemize}
    \item $A \times A'$ denotes the product automaton of $A$ and $A'$, that is,
    \[
    A \times A' = (\Sigma_k^m, Q \times Q', (q_0, q_0'), F \times F', \Delta''),
    \]
    where
    \[
    \Delta'' := \{((q, q'), \vv a, (q_1, q_1')) \mid (q, \vv a, q_1) \in \Delta
    \text{ and } (q', \vv a, q_1') \in \Delta'\};
    \]

    \item $\pr(A)$ denotes the finite automaton obtained from $A$ by removing the last digit from the label of each transition;

    \item $A^{+n}$ denotes the finite automaton
    $(\Sigma_k^{m+n}, Q, q_0, F, \Delta^{+n})$, where
    \[
    \Delta^{+n} := \{(q, (\vv a, \vv b), q') \in Q \times \Sigma_k^{m+n} \times Q
    \mid (q, \vv a, q') \in \Delta\};
    \]

    \item $\overline{A} := (\Sigma_k^m, Q, q_0, Q \setminus F, \Delta)$ denotes the
    automaton obtained from $A$ by complementing the set of accepting states;

    \item $\det(A)$ denotes the determinization of $A$, with $\mathcal P(Q)$ as
    the set of states.
\end{itemize}

The following facts are well known:
\begin{itemize}
    \item $L(A \times A') = L(A) \cap L(A')$;

    \item $
    L(\pr(A)) =
    \{u \in (\Sigma_k^{m-1})^*
    \mid \exists w \in L(A) \text{ such that } u = \pr(w)\}$, where $\pr(w)$ is obtained from $w$ by removing the last digit from each symbol of $w$;

    \item $L(A^{+n}) = L(A) \times (\Sigma_k^n)^*$;

    \item if $A$ is a DFA, then $L(\overline{A}) = \overline{L(A)}$;

    \item $L(\det(A)) = L(A)$ and $\det(A)$ is a DFA.
\end{itemize}

\section{Completeness of $\T_k$}\label{section_axiomatization}

This section is devoted to proving the following theorem.

\begin{theorem}\label[theorem]{main_theorem}
    The theory $\T_k$ axiomatizes $\BA_k$. In particular, $\BA_k$ admits a $\Pi_1$-axiomatization.
\end{theorem}

Before proving this theorem, we need some auxiliary results. The following lemma establishes some basic facts about $V_k$ and $(x)_d$. The scheme (Comp) is not used in this lemma.

\begin{lemma}\label[lemma]{lemma_for_coefs_in_expansion}
    The following are provable in $\T_k$:
    \begin{enumerate}[start=1,label={\rm (\roman*)}]
        \item $x > 0 \to P_k(V_k(x))$;
        \item $V_k(x) \leq x$;
        \item $x > 0 \to \bigvee_{a = 1}^{k - 1} \big((x)_{V_k(x)} = a \big)$;
        \item $\big(x > 0 \wedge P_k(d) \wedge d < V_k(x)\big) \to (x)_d = 0$;
        \item $\bigwedge_{a' \ne a} \Big ( (x > 0 \wedge (x)_d = a) \to \neg \big ((x)_d  = a' \big) \Big)$ for all $a \in \{0, \dots, k - 1\}$;
        \item $x \not\equiv_k 0 \to V_k(x) = \one$;
        \item $\big(d > \one \wedge P_k(d)\big) \to \existsP d' \  \big(d = k d' \big)$;
        \item $\big(x  > 0 \wedge y > 0 \wedge V_k(x) \leq V_k(y)\big) \to V_k(x + y) \geq V_k(x)$.
    \end{enumerate}
\end{lemma}

\begin{proof}
    We argue in $\T_k$. 
    \begin{enumerate}[(i)]
        \item Let $x > 0$. By (V3), there exist $y$ and $a \in\{1, \dots, k - 1\}$ such that $x = y + aV_k(x)$ and either $V_k(y) > V_k(x)$ or $y = 0$. This implies that $V_k(x) > 0$ since if $V_k(x) = 0$, then $x = y$ and $V_k(x) = V_k(y)$, which implies that $y = 0$ and, hence, $x = 0$, a contradiction. Next, by (V5), we have $V_k(V_k(x)) = V_k(x)$ and, since $V_k(x) > 0$, we obtain $P_k(V_k(x))$.

        \item The case $x = 0$ is immediate. Suppose $x > 0$. By (V3), we have $x \geq a V_k(x)$ for some $a \in \{1, \dots, k - 1\}$, and thus $x \geq V_k(x)$.
        
        \item This follows directly from (V3) together with part (i).
        
        \item Take $y = 0$ and $z = x$ in the definition of ``$(x)_d = a$''.

        \item Let $x > 0$ and $P_k(d)$. Suppose, towards a contradiction, that $(x)_d = a$ and $(x)_d = a'$ for distinct $a, a' \in \{0, \dots, k - 1\}$. Without loss of generality assume $a < a'$. Then $x$ can be represented as $x = y + ad + z$ and $x = y' + a' d + z'$, where $y, y' < d$, and either $z = 0$ or $V_k(z) > d$, and similarly
        either $z' = 0$ or $V_k(z') > d$. Consider the following cases.

        \begin{enumerate}
            \item Case $y = y'$. By cancellation,
            $$z = (a' - a)d + z'.$$
            This implies $z > 0$, hence $V_k(z) > d$. By (V1), we obtain $V_k((a' - a) d) = d$. If $z' = 0$, then $V_k(z) = d$, contradicting $V_k(z) > d$. If $z' > 0$, then $V_k(z') > d$ and by (V2) we have $V_k((a' - a) d + z' ) = V_k((a' - a) d) = d$, again contradicting $V_k(z) > d$.

            \item Case $y < y'$. Denote by $y' - y$ the unique (positive) element $w$ such that $w + y = y'$, which exists by (Ord1). By cancellation,
            $$z = (y' - y) + (a' - a)d + z'.$$
            Hence $z > 0$, so $V_k(z) > d$. By item (ii), $V_k(y' - y) \leq y' - y < d$, and by (V1), $V_k((a' - a)d) = d$. Thus, by (V2), $V_k((y' - y) + (a' - a) d) = V_k(y' - y) < d$. If $z' > 0$, then $V_k(z') > d$, and again by (V2), $V_k((y' - y) + (a' - a)d + z') < d$. If $z' = 0$, then we trivially have $V_k((y' - y) + (a' - a)d + z') < d$. In both cases this contradicts $V_k(z) > d$.

            \item Case $y' > y$. By cancellation,
            $$(y - y') + z = (a' - a)d + z'.$$
            By (ii), $V_k(y - y') < d$, while by (V1) and (V2) we have $V_k((a' - a)d + z') = d$ and $V_k((y - y') + z) = V_k(y - y') < d$, contradicting $V_k((a' - a)d + z') = d$.
        \end{enumerate}

        \item Let $x \not\equiv_k 0$. By (Mod) there exist $y$ and $a \in \{1, \dots, k - 1\}$ such that $x = k y + \underline a.$ By (V0) and (V1), we have
        $$V_k(\underline a) = V_k(a \one) = V_k(\one) = \one.$$
        If $y = 0$ we are done. Assume $y > 0$. Then by (V1) and (i), $V_k(k y) = k V_k(y) \geq k$, hence, by (V2), $$V_k(x) = V_k(ky + \underline a) = V_k(\underline a) = \one.$$

        \item Let $d > \one$ and $P_k(d)$. If $d \not\equiv_k 0$, then (vi) we obtain that $V_k(d) = \one < d$, which contradicts $P_k(d)$. Hence,  there is $d'$ such that $d = k d'$. We claim that $P_k(d')$. Indeed, 
        $$k d' = d = V_k(d) = V_k(k d') = kV_k(d'),$$
        where the last equality follows from (V1). Hence, we obtain $d' = V_k(d')$. Moreover, since $d > 0$, we have $d' > 0$. So $P_k(d')$ holds.

        \item Fix $x, y > 0$ such that $V_k(x) \leq V_k(y)$. If $V_k(x) < V_k(y)$, the statement follows immediately from (V2). Thus assume $V_k(x) = V_k(y)$ and suppose, towards a contradiction, that $V_k(x + y) < V_k(x)$. We first claim that $V_k((k - 1) z) = V_k(z)$ for any $z > 0$. Indeed, if $V_k(z) < V_k((k - 1)z)$, then by (V2),
        $$kV_k(z) = V_k(k z) = V_k((k - 1)z + z) = V_k(z),$$ a contradiction. If $V_k(z) > V_k((k - 1)z)$, then by (V2),
        $$kV_k(z) = V_k(k z) = V_k((k - 1)z + z) = V_k((k - 1)z) < V_k(z),$$
        a contradiction. This proves the claim. Applying the claim with $z = x + y$, we obtain $V_k((k - 1)(x + y)) = V_k(x + y) < V_k(x) = V_k(y)$. Hence, by two applications of (V2),
        $$V_k((k - 1)(x + y) + x + y) = V_k((k - 1)(x + y)) = V_k(x + y).$$
        Consequently, 
        $$k V_k(x + y) = V_k(k(x + y)) = V_k((k - 1)(x + y) + x + y) = V_k(x + y),$$
        and $k V_k(x + y) = V_k(x + y)$ which is impossible. This contradiction completes the proof.
    \end{enumerate}
\qed\end{proof}

The following proposition and corollary show that the least element principle on $P_k$ and the induction principle on $P_k$ are provable in $\T_k$. This is the only essential use of the scheme (Comp).

\begin{proposition}\label[proposition]{least_element_scheme}
    Let $\varphi(d, \vv y)$ be a $\Delta_0$-formula. Then the formula 
    $$\existsP d \ \varphi(d, \vv y) \to \existsP d \ \Big( \varphi(d, \vv y) \wedge \forallP d' < d  \ \neg \varphi(d', \vv y) \Big)$$
    is provable in $\T_k$.
\end{proposition}

\begin{proof}
    We argue in $\T_k$. Let $d \in P_k$ be such that $\varphi(d, \vv y)$ holds. By (Comp), there exists $x < kd$ such that for every $d' < kd$, $d' \in P_k$, we have $(x)_{d'} = 1$ if $\varphi(d', \vv y)$ and $(x)_{d'} = 0$ otherwise. Since $(x)_d = 1$, it follows that $x > 0$. Let $d_0 := V_k(x)$. By items (iii) and (v) of \cref{lemma_for_coefs_in_expansion}, we have $(x)_{d_0} = 1$, and therefore
    $\varphi(d_0, \vv y)$ holds. Moreover, by \cref{lemma_for_coefs_in_expansion}(iv), we have
    $(x)_{d'} = 0$ for all $d' \in P_k$ with $d' < d_0$, and hence
    $\neg \varphi(d', \vv y)$ for all such $d'$.
\qed\end{proof}

\begin{corollary}\label[corollary]{induction_scheme}
    Let $\varphi(d, \vv y)$ be a $\Delta_0$-formula. Then the formula 
    $$\varphi(\underline{1}, \vv y) \wedge \forallP d \ (\varphi(d, \vv y) \to \varphi(kd, \vv y)) \to \forallP d \ \varphi(d,\vv y)$$
    is provable in $\T_k$.
\end{corollary}

\begin{proof}
    The proof is standard; just apply \cref{least_element_scheme} and \cref{lemma_for_coefs_in_expansion}(vii).
\qed\end{proof}

We denote by $x|_d = y$ the following $\BA_k$-formula:
\begin{align*}
    P_k(d) \wedge (y < d) \wedge \exists z \leq x \ \Big ( x= y + z \wedge \big(z = 0 \vee V_k(z)\geq d \big ) \Big ).
\end{align*}

Intuitively, $x|_d$ represents the sum of the first $\log_k d$ summands in the $k$-expansion of $x$, namely
\[
    (x)_1 + (x)_k \cdot k + \dots + (x)_{d/k} \cdot (d/k).
\]
Equivalently, $x|_d$ can be described as the remainder of $x$ modulo $d$.

The following lemma provides us some basic properties of $x|_d$.

\begin{lemma}\label[lemma]{lemma_for_restriction}
    The following are provable in $\T_k$:
    \begin{enumerate}[start=1,label={\rm (\roman*)}]
        \item $\forallP d\ \forallP d' \ (d' \leq d \vee kd \leq d')$;
        \item $\forall x\  \forallP d \ \exists y \leq x \ (x|_d = y)$;
        \item $\forall x\ \forall y\ \forall y'\ \forallP d \  (x|_d = y \wedge x|_d = y' \to y = y')$;
        \item $\forall x\ \forallP d \ (x < d \to x|_d = x)$;
        \item $\forall x \ \forallP d \ \forallP d' < d \ \Big(\bigvee_{a = 0}^{k - 1}\big(  (x)_{d'} = a \wedge (x|_d)_{d'} = a \big)\Big) $.
    \end{enumerate}
\end{lemma}

\begin{proof}
    We argue in $\T_k$.
    \begin{enumerate}[(i)]
        \item Using \cref{induction_scheme}, we prove by induction on $d \in P_k$ the $\Delta_0$-formula
        $$\forallP d' < kd \ (d' \leq d).$$
        The base case $d = \one$ is immediate. For the inductive step, assume $\forallP d' < kd \ (d' \leq d)$, and let us show that $\forallP d' < k^2d \ (d' \leq kd)$. Suppose, towards a contradiction, that there exists $d' \in P_k$ such that $kd < d' < k^2 d$. By \cref{lemma_for_coefs_in_expansion}(vii), there exists $d'' \in P_k$ such that $d' = k d''$. This implies $d < d'' < kd$, contradicting the induction hypothesis. By \cref{induction_scheme}, the desired statement follows.

        \item Fix $x$. We prove by induction on $d \in P_k$ the $\Delta_0$-formula
        $$\exists y \leq x \ (x|_d = y).$$
        The base case $d = \one$ is immediate: take $y = 0$ and $z = x$ in the definition of $x|_d = y$. For the inductive step, assume $\exists y \leq x \ (x|_d = y)$ and let us prove $\exists y' \leq x \ (x|_{kd} = y')$. Fix $y \leq x$ such that $x|_d = y$. By the definition of $x|_d = y$, there exists $z \leq x$ such that $x = y + z$ and either $z = 0$ or $V_k(z) \geq d$. If $z = 0$, we may take $y' = y$, and we are done. Suppose $V_k(z) \ge d$. Then $z > 0$. By (V3) together with \cref{lemma_for_coefs_in_expansion}(iv), there exist $z'$ and $a \in \{0, \dots, k - 1\}$ such that $z = a d + z'$ and $z' = 0$ or $V_k(z') > d$. Now one can take $y' = y + ad$ and $z'$ as witnesses that $x|_{k d} = y'$.

        \item Fix $x, y, y'$ and $d$ such that $x|_d = y$ and $x|_d = y'$. Then there exist $z, z' \leq x$ such that 
        $$x= y + z \wedge \big(z = 0 \vee V_k(z)\geq d \big)\quad\text{and}\quad x= y' + z' \wedge \big(z' = 0 \vee V_k(z')\geq d \big).$$ 
        Without loss of generality, assume that $y \leq y'$. Then $z' \leq z$ and $y' - y = z - z'$. Suppose $y' - y \ne 0$. Then $0 < y' - y \leq y' < d$, and hence $V_k(y' - y) \leq y' - y < d$. We claim that $V_k(z - z') \ge d$. Indeed, if $V_k(z - z') < d$, then by (V2), $V_k(z) = V_k((z - z') + z') = V_k(z - z') < d$, a contradiction. Hence $V_k(z - z') \geq d$, contradicting $V_k(y' - y) < d$. Therefore $y' - y = 0$ and $y = y'$.

        \item This follows immediately from the definition of $x|_d = y$ by taking $z = 0$.

        \item Fix $x$, $d \in P_k$, and $d' \in P_k$ such that $d' < d$. By items (ii) and (iii), there exist unique $y$ and $z$ such that $x = y + z$, $y < d$ and $z = 0$ or $V_k(z) \geq d$, that is, $x|_d = y$. By \cref{lemma_for_coefs_in_expansion}, there exists a unique $a \in \{0, \dots, k - 1\}$ such that $(y)_{d'} = a$. Thus we can write $y = t + ad' + u$ for some $t < d'$ and with $u = 0$ or $V_k(u) > d'$. We claim that $(x)_{d'} = a$. Indeed, $x = y + z = t + ad' + u + z$. It is sufficient to show that $u + z = 0$ or $V_k(u + z) > d'$. The cases of $u = 0$ or $z = 0$ are immediate. Otherwise, assume $u, z > 0$, that is, $V_k(u) > d'$ and $V_k(z) \geq d > d'$. By \cref{lemma_for_coefs_in_expansion}(viii), it follows that $V_k(u + z) > d'$, as required.
    \end{enumerate}
\qed\end{proof}

The rest of this section is organized as follows. First, we express runs of automata in the language of $\BA_k$ and prove some natural facts about automata, formalized in $\T_k$ (\cref{lemma_for_automaton_constructions}). Next, we prove in $\T_k$ that the standard DFAs for $0$, $S$, $+$, $V_k$, and $\leq$ recognize the corresponding predicates (\cref{lemma_for_basic_automata} and \cref{corollary_base_case}). These results imply the formalized version of \cref{theorem_Buchi_Bruyere}, namely \cref{theorem_Buchi_formalized}, which is crucial for us. After these steps, one can easily derive \cref{main_theorem}. Now let us proceed with this plan.

Denote $\Sigma_k := \{0, \dots, k - 1\}$. Let $A = (\Sigma_k^m, Q, q_0, F, \Delta)$ be a finite automaton,
where $\Delta \subseteq Q \times \Sigma_k^m \times Q$
is the set of transitions.
For any $q_1 \in Q$, we define the following $\Delta_0$-formula $W_{A, q_1}(x_1, \dots, x_m, d)$:
\begin{align*}
     P_k(d) \wedge &\exists (w_q)_{q \in Q} < kd  \ \Big(  \forallP d' \leq d \bigvee_{q \in Q} \big((w_q)_{d'} = 1 \wedge \bigwedge_{q' \ne q} ((w_{q'})_{d'} = 0)\big) \wedge  \\
     &\big((w_{q_0})_{\one} = 1\big) \wedge \big((w_{q_1})_{d} = 1\big) \wedge\\
     &\forallP d' < d  \bigwedge_{q \in Q}\big((w_q)_{d'} = 1 \to \bigvee_{(q, \vv a, q') \in \Delta} \big((w_{q'})_{k d'} = 1 \wedge \bigwedge_{i = 1}^m (x_i)_{d'} = a_i\big)\big)
    \Big).
\end{align*}

Here, by $\exists (w_q)_{q \in Q} < kd \dots$ we mean a block of $|Q|$ bounded existential quantifiers. This formula expresses that the automaton $A$
can reach the state $q_1$ by reading the $k$-expansion of
$(x_1,\dots,x_m)$ in exactly $\log_k d$ steps. The numbers $w_q$ encode a run of the automaton: the condition
$(w_q)_{d'} = 1$ means that the state at step $\log_k d'$ of the run is $q$.
The first row of the formula ensures that $(w_q)_{q \in Q}$ encodes a
sequence of exactly $\log_k d + 1$ states, with exactly one active state
at each position.
The second row specifies that the run starts in the initial state $q_0$
and ends in the state $q_1$.
Finally, the third row guarantees that all transitions along the run are
consistent with the transition relation $\Delta$ and with the digits of
the $k$-expansion of $(x_1,\dots,x_m)$.

\begin{lemma}\label[lemma]{lemma_for_automaton_constructions}
    Let $A = (\Sigma_k^m, Q, q_0, F, \Delta)$ and $A' = (\Sigma_k^m, Q', q_0', F', \Delta')$ be finite automata. Then the following are provable in $\T_k$:
    \begin{enumerate}[start=1,label={\rm (\roman*)}]
        \item $\forall \vv x \ \forallP d \ \bigwedge_{q \in Q, q' \in Q'} \Big( W_{A \times A', (q, q')}(\vv x, d) \leftrightarrow W_{A, q}(\vv x, d) \wedge W_{A', q'} (\vv x, d) \Big)$;
        \item $\forall \vv x \ \forallP d \ \bigwedge_{q \in Q} \Big(W_{A, q}(\vv x, d) \leftrightarrow \bigvee\limits_{H \subseteq Q, q \in H} W_{\det(A), H}(\vv x, d) \Big)$;
        \item $\forall \vv x \ \forallP d \ \bigwedge_{q \in Q} \Big(W_{\pr(A), q}(\vv x, d) \leftrightarrow \exists y < d \ W_{A, q}(\vv x, y, d) \Big)$;
        \item[{\rm (iii')}] $ \forall \vv x \ \forallP d \ \bigwedge_{q \in Q} \Big(W_{\pr(A), q}(\vv x, d) \leftrightarrow \exists y \ W_{A, q}(\vv x, y, d) \Big)$;
        \item $ \forall \vv x \ \forall \vv y \ \forallP d \ \bigwedge_{q \in Q} \Big(W_{A, q}(\vv x, d) \leftrightarrow W_{A^{+n}, q}(\vv x, \vv y, d) \Big)$;
    \end{enumerate}
    Moreover, if $A$ is a DFA, then the following is provable in $\T_k$:
    \begin{enumerate}[\rm(v)]
        \item $ \forall \vv x \ \forallP d \ \bigvee_{q \in Q}  \Big(W_{A, q}(\vv x, d) \wedge \bigwedge_{q' \ne q} \neg W_{A, q'}(\vv x, d) \Big)$.
    \end{enumerate}
\end{lemma}

\begin{proof}
    The proof of (i), (ii), (iii), (iv), (v) is straightforward. Fix $\vv x$ (and $\vv y$ in (iv)) and apply induction on $d$. Statement (iii') follows from (iii), since 
    $$W_{A, q}(\vv x, y, d) \leftrightarrow W_{A, q}(\vv x, y|_d, d)$$
    is provable in $\T_k$ by \cref{lemma_for_restriction}(v). We prove (iii) as an illustration in Appendix \ref{appendix_1}.
\qed\end{proof}

Now construct DFAs $A_0$, $A_S$, $A_+$, $A_{V_k}$, and $A_\leq$ recognizing the predicates
$$x = 0, \quad S(x) = y, \quad x + y = z, \quad V_k(x) = y, \quad x \leq y,
$$
respectively (see below). Accepting states are indicated by double circles. In each automaton, there is an implicit sink state $q_{\mathrm{sink}}$ required to make the transition function total (it is omitted from the diagrams to save space). Unless stated otherwise, $a$ and $b$ range over $\{0, \dots, k - 1\}$.

\tikzset{
  sink/.style={
    state,
    dashed,
    draw,
    minimum size=18pt
  },
  sink edge/.style={
    dashed
  }
}

\underline{Automaton $A_0$}: \ \ \ \ \ \ \ \ \ \ \ \ \ \ \ \underline{Automaton $A_S$}:

\begin{tikzpicture}[->,>=Stealth,auto,node distance=3cm]
  \node[state,initial,accepting] (q0) {$q_0$};

  \path
    (q0) edge[loop above] node {\scriptsize $0$} ();

  \node[state,initial,right=3cm of q0] (q00) {$q_0$};
  \node[state,accepting,right=4cm of q00] (q11) {$q_1$};

  \path
    (q00) edge[loop above] node {\scriptsize $(k-1,0)$} ()
         edge[bend left=15] node[above] {\scriptsize $(a,a+1),\ a \in \{0, \dots, k - 2\}$} (q11)

    (q11) edge[loop above] node {\scriptsize $(a,a)$} ();
\end{tikzpicture}

\underline{Automaton $A_+$}:

\begin{tikzpicture}[->,>=Stealth,auto,node distance=4.5cm]
  \node[state,initial,accepting] (q0) {$q_0$};
  \node[state,right=7cm of q0] (q1) {$q_1$};

  \path
    (q0) edge[loop above] 
          node {\scriptsize $(a,b,a+b),\ a+b<k$} ()
         edge[bend left=18] 
          node[above] {\scriptsize$(a,b,a+b-k),\ a+b\ge k$} (q1)

    (q1) edge[loop above] 
          node {\scriptsize$(a,b,a+b+1-k),\ a+b+1\ge k$} ()
         edge[bend left=18] 
          node[below] {\scriptsize$(a,b,a+b+1),\ a+b+1<k$} (q0);

\end{tikzpicture}

\underline{Automaton $A_{V_k}$}: \ \ \ \ \ \ \ \ \ \ \ \ \ \ \ \ \ \ \ \ \ \ \ \ \ \ \ 
\underline{Automaton $A_\leq$}:

\begin{tikzpicture}[->,>=Stealth,auto,node distance=4cm]
  \node[state,initial,accepting] (q0) {$q_0$};
  \node[state,accepting,right=2cm of q0] (q1) {$q_1$};

  \path
    (q0) edge[loop above] node {\scriptsize $(0,0)$} ()
         edge[bend left=15] node[above, align=left] {\scriptsize $(a,1), \ a > 0$} (q1)

    (q1) edge[loop above] node {\scriptsize $(a,0)$} ();


  \node[state,initial,accepting, right=2cm of q1] (q00) {$q_0$};
  \node[state,right=2cm of q00] (q11) {$q_1$};

  \path
    (q00) edge[loop above] node {\scriptsize $(a,b),\ a \leq b$} ()
         edge[bend left=18] node[above] {\scriptsize $(a,b),\ a>b$} (q11)

    (q11) edge[loop above] node {\scriptsize $(a,b),\ a\ge b$} ()
         edge[bend left=18] node[below] {\scriptsize $(a,b),\ a<b$} (q00);
\end{tikzpicture}

\begin{lemma}\label[lemma]{lemma_for_basic_automata} The following are provable in $\T_k$:
    \begin{enumerate}[start=1,label={\rm (\roman*)}]
        \item $ \forall x \ \forallP d \ \Big(x|_d = 0 \leftrightarrow W_{A_0, q_0}(x, d) \Big)$;
        \item
        $ \forall  x\  \forall y \ \forallP d  \Big ( \big(x|_d = d - 1 \wedge y|_d = 0 \leftrightarrow W_{A_S, q_0}(x, y, d) \big) \wedge \big (S(x|_d) = y|_d \leftrightarrow W_{A_S, q_1}(x, y, d)\big) \Big)$;

        \item $ \forall  x \ \forall y \ \forall z \ \forallP d \ \Big ( \big(x|_d + y|_d = z|_d \leftrightarrow W_{A_+, q_0}(x, y, z, d) \big) \wedge \big(x|_d + y|_d = z|_d + d \leftrightarrow W_{A_+, q_1}(x, y, d) \big) \Big)$;

        \item $ \forall x \ \forall y \ \forallP d \ \Big ( \big(x|_d = 0 \wedge y|_d = 0 \leftrightarrow W_{A_{V_k}, q_0}(x, y, d) \big) \wedge \big(V_k(x|_d) = y|_d \leftrightarrow W_{A_{V_k}, q_1}(x, y, d) \big) \Big)$;

        \item $ \forall  x \ \forall y \ \forallP d \ \Big ( \big(x|_d \leq y|_d \leftrightarrow W_{A_{\leq}, q_0}(x, y, d) \big) \wedge  \big(x|_d > y|_d \leftrightarrow W_{A_{\leq}, q_1}(x, y, d) \big) \Big)$.
    \end{enumerate}
\end{lemma}

\begin{proof}
    The proof is standard and proceeds by induction on $d$. For an illustration, see the proof of (iii) in Appendix \ref{appendix_2}.
\qed\end{proof}

Next, we want to get rid of the quantifiers on $d$ in the formulas from \cref{lemma_for_basic_automata}. In order to do this, let us introduce the function $g_n : \mathbb{N}^n \to \mathbb{N}$ that maps $(x_1, \dots, x_n)$ to the least power of $k$ strictly greater than each of $x_1, \dots, x_n$. Clearly, the graph of $g_n(x_1, \dots, x_n) = d$ is definable by the quantifier-free formula
$$\Big( \bigwedge_{i = 1}^n (x_i = 0) \wedge d = \one \Big) \vee \Big(P_k(d) \wedge \bigvee_{i = 1}^n \big( \bigwedge_{j = 1}^n (x_i \geq x_j) \wedge (x_i < d) \wedge (d \leq k x_i) \big) \Big).$$
Hence all occurrences of $g_n$ can be eliminated using a bounded quantifier. We will write simply $g$ instead of $g_n$ when the arity is clear from the context. Note that by \cref{lemma_for_restriction}(iv) we have $$\T_k \vdash \forall \vv x \ \forallP d \ (d \geq g(\vv x) \to \bigwedge_{i = 1}^n (x_i|_d = x_i)).$$

\begin{corollary}\label[corollary]{corollary_base_case} The following are provable in $\T_k$:
    \begin{enumerate}[start=1,label={\rm (\roman*)}]
        \item $\forall x \ \Big (x = 0 \leftrightarrow W_{A_0, q_0}\big(x, g(x)\big) \Big)$;

        \item  $ \forall x \ \forall y \ \Big(S(x) = y \leftrightarrow  W_{A_S, q_1}\big(x, y, g(x, y)\big)\Big)$;

        \item $\forall x \ \forall y \ \forall z \ \Big(x + y = z \leftrightarrow W_{A_+, q_0}\big(x, y, z, g(x, y, z)\big)\Big)$;

        \item  $\forall x \ \forall y \ \Big(V_k(x) = y \leftrightarrow \bigvee\limits_{i = 1, 2}W_{A_{V_k}, q_i}\big(x, y, g(x, y)\big) \Big)$;

        \item  $\forall x \ \forall y \ \Big(x \leq y \leftrightarrow W_{A_{\leq}, q_0}\big(x, y, g(x, y)\big)\Big)$.
    \end{enumerate}
\end{corollary}

\begin{proof}
    Substitute $g(\dots)$ for $d$ in \cref{lemma_for_basic_automata}.
\qed\end{proof}

The following result is crucial for us.

\begin{theorem}\label[theorem]{theorem_Buchi_formalized}
    For any formula $\varphi(\vv x)$, there exists a DFA $A$ such that
    $$\T_k \vdash \forall \vv x \ \Big(\varphi(\vv x) \leftrightarrow \bigvee_{q \in F} W_{A, q}\big(\vv x, g(\vv x) \big) \Big).$$
    Here, $F$ denotes the set of accepting states of $A$.
\end{theorem}

\begin{proof}
    We may assume that all atomic formulas are simple (i.e., they contain no nested occurrences of function symbols) and that the only logical connectives are $\neg$, $\wedge$, and $\exists$.

    The proof proceeds by induction on $\varphi$. The base case of atomic formulas is handled in \cref{corollary_base_case}.

    Let $\varphi = \neg \psi$ and let $A$ be a DFA for $\psi$, which exists by the induction hypothesis. Then, by \cref{lemma_for_automaton_constructions}(v) $\overline{A}$ is a DFA for $\psi$. 

    Let $\varphi = (\psi_0 \wedge \psi_1)$. By the induction hypothesis, there exist DFAs $A_0$ and $A_1$ for $\psi_0$ and $\psi_1$, respectively. After possibly adding dummy variables and applying \cref{lemma_for_automaton_constructions}(iv), we may assume that $\varphi$, $\psi_0$, and $\psi_1$ share the same set of free variables and that $A_0$ and $A_1$ share the same alphabet (namely, $\Sigma_k^m$, where $\varphi$ has $m$ free variables). Hence, by \cref{lemma_for_automaton_constructions}(i), the product $A_0 \times A_1$ is a DFA for $\varphi$.


    Finally, let $\varphi = \exists y\ \psi$ and let $A$ be a DFA for $\psi$. If $y$ does not occur freely in $\psi$, then $A$ is already a DFA for $\varphi$. If $y$ occurs freely in $\psi$, then by \cref{lemma_for_automaton_constructions}(ii, iii'), $\det(\pr(A))$ is a DFA for $\varphi$.
\qed\end{proof}


The following proposition is proved in the standard manner, see \cite[Theorem 1.8]{hajek_pudlak_2017}.

\begin{proposition}\label[proposition]{prop_Delta_0_completeness}
    If $\varphi$ is a $\Delta_0$-sentence and $\N \vDash \varphi$, then $\T_k \vdash \varphi$.
\end{proposition}

\emph{Proof of \cref{main_theorem}.} Let $\varphi$ be a sentence, such that $\N \vDash \varphi$. By \cref{theorem_Buchi_formalized}, there is $\Delta_0$-sentence $\theta$ such that $\T_k \vDash \varphi \leftrightarrow \theta$ and, by \cref{prop_Delta_0_completeness}, $\T_k \vdash \theta$. Hence, $\T_k \vdash \varphi$. \qed

\begin{remark}\label[remark]{remark_on_induction}
    It is clear from the proof that the only essential use of the comprehension schema is to derive the least element schema. Hence, $\BA_k$ can alternatively be axiomatized using the induction scheme from \cref{induction_scheme} together with the axioms (Add), (Ord0), (Ord1), (Mod), (V0), (V1), (V2), (V3), (V4), and (V5).
\end{remark}

However, it remains unknown whether $\BA_k$ is finitely axiomatizable. We leave the following questions open.

\begin{question}
    \begin{enumerate}[(1)]
        \item Is $\BA_k$ finitely axiomatizable?
        \item Is $\BA_k$ finitely axiomatizable over Presburger Arithmetic?

        \item Is it possible to replace the induction scheme for $\Delta_0$-formulas with a quantifier-free induction scheme in the axiomatization of $\BA_k$?
        \item Does $\BA_k$ admit a \emph{finite} expansion by definable functions or predicates that admits  quantifier elimination?
    \end{enumerate}
\end{question}


\begin{subappendices}
\renewcommand{\thesection}{\Alph{section}}%

\section{Proof of \cref{lemma_for_automaton_constructions}(iii)}\label{appendix_1}

{\bf \cref{lemma_for_automaton_constructions}(iii). } {\it
    Let $A = (\Sigma_k^m, Q, q_0, F, \Delta)$ and $A' = (\Sigma_k^m, Q', q_0', F', \Delta')$ be finite automata. Then the following is provable in $\T_k$:
    \begin{enumerate}[start=3,label={\rm (\roman*)}]
        \item $\forall \vv x \ \forallP d \ \bigwedge_{q \in Q} \Big(W_{\pr(A), q}(\vv x, d) \leftrightarrow \exists y < d \ W_{A, q}(\vv x, y, d) \Big)$.
    \end{enumerate}
}

\begin{proof}
    
    We argue in $\T_k$. Fix $\vv x$ and argue by induction on $d$:
    $$\bigwedge_{q \in Q} \Big(W_{\pr(A), q}(\vv x, d) \leftrightarrow \exists y < d \ W_{A, q}(\vv x, y, d) \Big).$$
    
    Consider the base case of $d = \one$. If $W_{\pr(A), q}(\vv x, d)$ holds, then by definition of $W_{\pr(A),q}$ we must have $q = q_0$. Hence we have $W_{A, q}(\vv x, 0, d)$. The backward implication is immediate from the definitions.

    For the inductive step, suppose 
    $$\bigwedge_{q \in Q} \Big(W_{\pr(A), q}(\vv x, d) \leftrightarrow \exists y < d \ W_{A, q}(\vv x, y, d) \Big).$$
    Denote by $\vv a$ the tuple $\Big( (x_i)_d \Big)_{i = 1,  \dots, |\vv x|}$. If $W_{\pr(A), q}(\vv x, kd)$ holds, then the definition of $W_{\pr(A), q}$ implies that $W_{\pr(A), q'}(\vv x, d)$ for some $q' \in Q$ such that $(q', \vv a, q)$ is a transition of $\pr(A)$. By the induction hypothesis there exists $y < d$ such that $W_{A, q}(\vv x, y, d)$. The definition of $\pr(A)$ implies that there is $b \in \{0, \dots, k - 1\}$ such that $(q', (\vv a, b), q)$ is a transition of $A$. Now it is easy to see that $W_{A, q}(\vv x, y + b d, kd)$ holds. 
    
    The backward implication is immediate, since if $(w_q)_{q \in Q}$ are witnesses for $W_{A, q}(\vv x, y, d)$, then the same $(w_q)_{q \in Q}$ witness for $W_{\pr(A), q}(\vv x, d)$.
\qed\end{proof}

\section{Proof of \cref{lemma_for_basic_automata}(iii)}
\label{appendix_2}

{\bf \cref{lemma_for_basic_automata}(iii). } {\it
    The following is provable in $\T_k$:
    \begin{enumerate}[start=3,label={\rm (\roman*)}]

        \item $ \forall  x \ \forall y \ \forall z \ \forallP d \ \Big ( \big(x|_d + y|_d = z|_d \leftrightarrow W_{A_+, q_0}(x, y, z, d) \big) \wedge \big(x|_d + y|_d = z|_d + d \leftrightarrow W_{A_+, q_1}(x, y, d) \big) \Big)$.


    \end{enumerate}

}

\begin{proof}
    As before, we argue in $\T_k$. Fix $x, y, z$. The base case of $d = \one$ is trivial, since $x|_{\one} = y|_{\one} = z|_{\one} = 0$ and by definition $W_{A_+, q_0}(x, y, z, \one)$ holds. At the same time, $W_{A_+, q_1}(x, y, z, \one)$ does not hold. For the inductive step, suppose that 
    $$\big(x|_d + y|_d = z|_d \leftrightarrow W_{A_+, q_0}(x, y, z, d) \big) \wedge  \big(x|_d + y|_d = z|_d + d \leftrightarrow W_{A_+, q_1}(x, y, d) \big).$$

    We first show that $x|_{kd} + y|_{kd} = z|_{kd} \leftrightarrow W_{A_+, q_0}(x, y, z, kd)$.
    Suppose $x|_{kd} + y|_{kd} = z|_{kd}$. We need to prove that $W_{A_+, q_0}(x, y, z, kd)$ holds. There exist unique $a, b, c \in \{0, \dots, k - 1\}$ such that 
    $$x|_{kd} = x|_d + a d, \quad y|_{kd} = y|_d + b d, \quad z|_{kd} = z|_d + c d,$$
    where $a = (x)_d$, $b = (y)_d$, and $c = (z)_d$. 
    Then we have
    $$z|_d + cd = z|_{kd} = x|_{kd} + y|_{kd} = x|_d + y|_d + (a + b)d.$$
    We claim that either $c = a + b$ or $c = a + b + 1$. Suppose, for a contradiction, that $c \leq a + b - 1$ or $c \geq a + b + 2$. In the case $c \leq a + b - 1$ we have 
    $$x|_d + y|_d + (a + b)d \geq (c + 1)d > z|_d + cd,$$
    a contradiction. In the case $c \geq a + b + 2$ we have
    $$z|_d + cd \geq (a + b + 2)d > x|_d + y|_d + (a + b)d,$$
    also a contradiction. Hence, $c = a + b$ or $c = a + b + 1$. In the first case, we have $x|_d + y|_d = z|_d $, and by the induction hypothesis, $ W_{A_+, q_0}(x, y, z, d)$ holds. Consequently, $W_{A_+, q_0}(x, y, z, kd)$ holds as well. The other forward implication is treated in a similar manner.

    Now assume that $W_{A_+, q_0}(x, y, z, kd)$,holds. We prove that $x|_{kd} + y|_{kd} = z|_{kd}$. Choose $a, b, c \in \{0, \dots, k - 1\}$ such that $x|_{kd} = x|_d + a d$, $y|_{kd} = y|_d + b d$ and $z|_{kd} = z|_d + c d$. $W_{A_+, q_0}(x, y, z, kd)$ implies that either $W_{A_+, q_0}(x, y, z, d)$ and $c = a + b$ or $W_{A_+, q_1}(x, y, z, d)$ and $c = a + b + 1$. In the first case, by the induction hypothesis we have $x|_d + y|_d = z|_d$. Therefore, $(x|_d + a d) + (y|_d + bd) = (z|_d + cd)$, which shows that $x|_{kd} + y|_{kd} = z|_{kd}$. The second case is treated in an analogous way.

    We now show that $x|_{kd} + y|_{kd} = z|_{kd} + kd \leftrightarrow W_{A_+, q_1}(x, y, kd)$.
    Assume first that $x|_{kd} + y|_{kd} = z|_{kd} + kd$. We prove that $W_{A_+, q_1}(x, y, z, kd)$ holds. Choose $a, b, c \in \{0, \dots, k - 1\}$ such that $x|_{kd} = x|_d + a d$, $y|_{kd} = y|_d + b d$ and $z|_{kd} = z|_d + c d$. Then we have 
    $$z|_d + (c + k)d = z|_{kd} + k d = x|_{kd} + y|_{kd} = x|_d + y|_d + (a + b)d.$$
    We claim that either $c = a + b - k$ or $c = a + b + 1 - k$. Suppose, towards a contradiction, that $c \leq a + b - 1 - k$ or $c \geq a + b + 2 - k$. If $c \leq a + b - 1 - k$, then 
    $$x|_d + y|_d + (a + b)d \geq (c + 1 + k)d > z|_d + (c + k)d,$$
    which contradicts the above equality. If $c \geq a + b + 2 - k$, then
    $$z|_d + (c + k)d \geq (a + b + 2)d > x|_d + y|_d + (a + b)d,$$
    again a contradiction. Hence, $c = a + b - k$ or $c = a + b + 1 - k$. In the first case, we have $x|_d + y|_d = z|_d $ and by the induction hypothesis $W_{A_+, q_0}(x, y, z, d)$ holds. Therefore, $W_{A+, q_1}(x, y, z, kd)$. The second case is treated analogously.

    The implication $W_{A_+, q_1}(x, y, z, kd) \to x|_{kd} + y|_{kd} = z|_{kd} + kd$ is proved in the same way as $W_{A_+, q_0}(x, y, z, kd) \to x|_{kd} + y|_{kd} = z|_{kd}$.
\qed\end{proof}
\end{subappendices}


\begin{credits}

\subsubsection{Acknowledgments.}
This work was supported by the Theoretical Physics and Mathematics Advancement Foundation <<BASIS>>.

\subsubsection{\discintname}
The author has no competing interests to declare that are
relevant to the content of this article.

\end{credits}
%
%
%
%

\end{document}